\documentclass[11pt,reqno]{amsart}
\usepackage{bm,amsmath,amsthm,amssymb,mathtools,verbatim,amsfonts,tikz-cd,mathrsfs}
\usepackage{enumitem}

\usepackage[hidelinks,colorlinks=true,linkcolor=blue, citecolor=black,linktocpage=true]{hyperref}
\usepackage{graphicx}
\usepackage[alphabetic]{amsrefs}

\usepackage[top=1.4in, bottom=1.2in, left=1.4in, right=1.4in,marginpar=1in]{geometry}

\usetikzlibrary{matrix,arrows,decorations.pathmorphing}

\usepackage{chngcntr}
\counterwithin{figure}{section}

\newtheorem{thm}{Theorem}[section]
\newtheorem{prop}[thm]{Proposition}

\newtheorem{cor}[thm]{Corollary}
\newtheorem{lem}[thm]{Lemma}

\theoremstyle{definition}
\newtheorem{define}[thm]{Definition}

\theoremstyle{remark}
\newtheorem{rem}[thm]{Remark}
\newtheorem{example}[thm]{Example}

\newcommand{\T}{\mathbb{T}}

\renewcommand{\subset}{\subseteq}

\renewcommand{\tilde}{\widetilde}

\DeclareMathOperator{\spin}{{spin}}
\DeclareMathOperator{\Spin}{{Spin}}

\DeclareMathOperator{\Sym}{{Sym}}

\newcommand{\bC}{\mathbb{C}}

\newcommand{\bF}{\mathbb{F}}

\newcommand{\bR}{\mathbb{R}}

\newcommand{\bZ}{\mathbb{Z}}

\newcommand{\alphas}{\boldsymbol \alpha}
\newcommand{\betas}{\boldsymbol \beta}

\newcommand{\cH}{\mathcal{H}}

\newcommand{\cL}{\mathcal{L}}

\newcommand{\CF}{\mathit{CF}}
\newcommand{\CFhat}{\widehat{\CF}}

\newcommand{\HF}{\mathit{HF}}

\newcommand{\HFR}{\mathit{HFR}}
\newcommand{\HFRhat}{\widehat{\HFR}}

\newcommand{\CFK}{\mathit{CFK}}
\newcommand{\HFK}{\mathit{HFK}}

\newcommand\CFKhat{\widehat{\CFK}}

\newcommand\HFhat{\widehat{\HF}}
\newcommand\HFKhat{\widehat{\HFK}}

\newcommand \HFKR{ \widehat{\mathit{HFKR}}}

\usepackage{leftidx}

\numberwithin{equation}{section}

\newcommand\co{\colon}

\allowdisplaybreaks

\title{A note on real Heegaard Floer homology and localization}

\author{Kristen Hendricks}
\thanks{KH was partially supported by NSF CAREER grant DMS-2019396 and NSF grant DMS-2505573.}
\address{Department of Mathematics, Rutgers University, New Brunswick, NJ, USA}

\begin{document}

\begin{abstract} We prove the existence of a localization spectral sequence for the hat variant of Guth and Manolescu's recent construction of real Heegaard Floer homology, and apply it to branched double covers and strongly invertible knots. Our construction applies to real Lagrangian Floer homology in exact symplectic manifolds equipped with anti-symplectic involutions more generally, and may be of independent interest to symplectic geometers.
\end{abstract}

\maketitle

\tableofcontents

\section{Introduction}

In the past few years, there has been considerable interest in gauge-theoretic invariants of three-manifolds and four-manifolds equipped with an involution, called real manifolds. In the Seiberg-Witten case, this line of study originated with Tian-Wang \cite{TW:real} and Nakamura \cite{Nakamura:real, Nakamura:intersection} and has had numerous recent remarkable applications, including for example \cite{Kato, Miyazawa, KPT:cables, HKM:branched, Baraglia}.

Ozsv{\'a}th and Szab{\'o}'s Heegaard Floer homology is a Lagrangian Floer theoretic counterpart of Seiberg-Witten theory \cite{OSDisks, OSProperties}. Recently, Guth and Manolescu \cite{GM:real} gave a construction of real Heegaard Floer homology associated to a three-manifold $Y$ with an orientation-preserving involution $\tau$ whose fixed point set has codimension two; such a pair $(Y, \tau)$ will be called a \emph{real three-manifold}. More generally, they give a construction of real Lagrangian Floer theory, as we now explain. 

Suppose that $(M, \omega)$ is a symplectic manifold with two Lagrangians $L_0$ and $L_1$, satisfying suitable hypotheses for defining Lagrangian Floer homology. (We will detail our hypotheses explicitly in the theorems that follow and in Section~\ref{sec:realhf}.) Suppose further that $M$ admits an anti-symplectic involution $R$ which exchanges $L_0$ and $L_1$; that is, $L_1=R(L_0)$. In this case, the fixed point set $M^R$ is also a Lagrangian submanifold of $M$. In its simplest description, the real Lagrangian Floer homology of $(M, L_0, L_1)$ defined in \cite{GM:real} is
\[ \HFR(M, L_0, L_1) := \HF(M, L_0, M^R). \]
Here the right-hand side is the ordinary Lagrangian Floer homology of the two Lagrangians $L_0$ and $M^R$ in $M$. The application to the case of Heegaard Floer homology then proceeds as follows. Recall that the Heegaard Floer homology of a 3-manifold $Y$ is computed by choosing a Heegaard diagram $(\Sigma, \alphas, \betas, w)$ such that $\Sigma$ is a surface of genus $g$ with two sets of attaching curves $\alphas$ and $\betas$ specifying handlebodies $U_{\alpha}$ and $U_{\beta}$ such that $Y= U_{\alpha}\cup_{\Sigma} U_{\beta}$, with $w$ denoting a basepoint away from the attaching curves. Heegaard Floer homology has several flavors, which are given by suitable versions of the Lagrangian Floer homology of $(\Sym^g(\Sigma), \T_{\alphas}, \T_{\betas})$, where $\Sym^g(\Sigma)$ denotes the symmetric product of the surface together with a symplectic form which is a product form away from the diagonal, and $\T_{\alphas}$ and $\T_{\betas}$ are the tori formed by taking the products of the two sets of attaching curves. In particular, $\HFhat(Y)$ is given by $\HF(\Sym^g(\Sigma - \{w\}), \T_{\alphas}, \T_{\betas})$. The construction splits along the set of $\spin^c$ structures on $Y$, so that $\HFhat(Y) = \oplus_{\mathfrak s \in \Spin^c(Y)} \HFhat(Y, \mathfrak s)$. Given a real three-manifold $(Y,\tau)$, Guth and Manolescu \cite[Proposition 3.2]{GM:real} use work of Nagase \cite[Proposition 2.4]{Nagase} to show there exists a Heegaard diagram $H = (\Sigma, \alphas, \betas, w)$ for $Y$ which is symmetric with respect to the action in the sense that that $\tau$ restricts to an orientation-reversing action $R$ on $\Sigma$ exchanging $\alpha$ curves with $\beta$ curves and fixing the basepoint. This means that $\tau$ exchanges the two handlebodies in the decomposition $Y = U_{\alpha} \cup_{\Sigma} U_{\beta}$ determined by the Heegaard diagram. They call this data $(H,R)$ a real Heegaard diagram for $(Y, \tau)$. Under these circumstances the induced map, also called $R$, on $\HF(\Sym^g(\Sigma), \T_{\alphas}, \T_{\betas})$ is anti-symplectic on the symmetric product and exchanges the Lagrangians. One may therefore take
\[ \HFRhat(Y, \tau) := HF(\Sym^g(\Sigma - \{w\}), \T_{\alphas}, \Sym^g(\Sigma - \{w\})^R)\]
and similarly for other variants. The construction splits along $\spin^c$ structures and further splits over real $\spin^c$ structures $\mathfrak s_r$, which we discuss in Section~\ref{sec:realhf}.

As detailed above, the case of real Lagrangian and Heegaard Floer theory deals with an anti-symplectic involution exchanging the Lagrangians. Let us contrast it to the situation of a symplectic involution $S$ on $(M, \omega)$ which fixes each of a pair of Lagrangians $L_0$ and $L_1$ setwise, which has been extensively studied in Lagrangian and Heegaard Floer homology. In this situation the fixed sets $(M^S, L_0^S, L_1^S)$, if connected, are a new triple consisting of a symplectic manifold with Lagrangians, such that one may consider $\HF(M^S, L_0^S, L_1^S)$. A major technical tool in this situation is the localization theorem of Seidel-Smith \cite{SS10} and Large \cite{Large}, which states that under nontrivial K-theoretic hypotheses, detailed in Section~\ref{sec:Large}, one has a spectral sequence
\[\HF(M, L_0, L_1) \otimes \bF[\theta, \theta^{-1}] \rightrightarrows \HF(M^S, L_0^S, L_1^S) \otimes \bF[\theta, \theta^{-1}]. \]
Here and throughout, $\bF$ is the field of two elements and $\theta$ has degree one, so that $H^*(B \bZ/2\bZ; \bF) \simeq \bF[\theta]$. It follows that there is a dimension inequality between the Floer homology of $(M, L_0, L_1)$ and $(M^S, L_0^S, L_1^S)$; this is the analog of the classical theorem of P. A. Smith for topological actions \cite{Smith:dimension}, as reformulated by Borel \cite{Borel:nouvelle}. This theorem has been applied broadly to Heegaard Floer theory via the following basic strategy: for some pair $(Y,L)$ consisting of a 3-manifold $Y$ with a possibly empty link $L$ and an orientaton-preserving action $\tau$ on $(Y,L)$, one constructs a Heegaard diagram for $(Y,L)$ such that the restriction of $\tau$ to the Heegaard surface $\Sigma$ is orientation-preserving and $\tau$ exchanges $\alpha$ curves with $\alpha$ curves and $\beta$ curves with $\beta$ curves, so that $\tau$ preserves each of $U_{\alpha}$ and $U_{\beta}$. One obtains the symplectic situation described above; that is, the induced map on $(\Sym^g(S), \T_{\alphas}, \T_{\betas})$ is symplectic on $\Sym^g(S)$ and preserves each of $\T_{\alphas}$ and $\T_{\betas}$ setwise. The first application of this, originally proved by the author \cite{Hendricks12:dcov-localization} and refined by the author  with Lipshitz and Sarkar \cite{HLS-equivariant}, was a spectral sequence
\begin{equation} \label{eqn:original}\HFKhat(\Sigma(K), \widetilde{K}, \mathfrak s_0) \otimes \bF[\theta, \theta^{-1}] \rightrightarrows \HFKhat(S^3,K).\end{equation}
Here $\Sigma(K)$ is the branched double cover of a knot $K$ in $S^3$, the knot $\widetilde{K}$ is the lift of $K$ in the branched double cover, $\mathfrak s_0$ is the unique spinnable $\spin^c$ structure on the branched double cover, and $\HFKhat(Y,K)$ is the knot version of Heegaard Floer homology for a nullhomologous knot in a 3-manifold $Y$, due to Ozsv{\'a}th-Szab{\'o} \cite{OSKnots, OSLinks} and J. Rasmussen \cite{RasmussenKnots}. Similarly, the author, Lipshitz, and Lidman \cite{HLL:branched} showed that if $L \subset Y$ is an $\ell$-component $\mathbb Z/2\mathbb Z$-nullhomologous link in $Y$ a three-manifold, $\Sigma(Y,L) \xrightarrow{\pi} Y$ is a branched cover of $(Y,L)$, and $\mathfrak s$ is a $\spin^c$ structure on $\Sigma(Y,L)$ which is the pullback $\pi^{*}\mathfrak s'$ of some $\spin^c$ structure $\mathfrak s'$ on $Y$, there is a similarly-constructed spectral sequence
\begin{equation}\label{eqn:original-manifolds} \HFhat(\Sigma(Y,L), \mathfrak s) \otimes H^*((S^1)^{\ell -1}) \otimes  \bF[\theta, \theta^{-1}] \rightrightarrows \bigoplus_{\substack{\mathfrak s'\in \Spin^c(Y) \\ \pi^* \mathfrak s' = \mathfrak s}} \HFhat(Y, \mathfrak s') \otimes H^*((S^1)^{\ell -1}) \otimes \bF[\theta, \theta^{-1}].\end{equation}
Here the sum on the right-hand side is taken over all $\spin^c$ structures on $Y$ whose pullback to $\Sigma(Y,L)$ is $\mathfrak s$. In particular, this implies that if $\Sigma(Y,L)$ is a Heegaard Floer $L$-space, which is to say a rational homology sphere with $\dim (\HFhat(Y, \mathfrak s)) = 1$ in each $\spin^c$ structure $\mathfrak s$, then so is $Y$ \cite[Corollary 1.2]{HLL:branched}.

In Remark 6.6 of their paper \cite{GM:real}, Guth and Manolescu conjecture that there is a localization spectral sequence for the hat variant of their real Heegaard Floer homology. To wit, if $\tau$ is the covering involution on $\Sigma(K)$, they expect that for each $\spin^c$ structure $\mathfrak s \in \Spin^c(\Sigma(K))$ there is a spectral sequence
\[ \HFhat(\Sigma(K), \mathfrak s) \otimes \bF[\theta, \theta^{-1}] \rightrightarrows \widehat{HFR}(\Sigma(K), \mathfrak s, \tau) \otimes \bF[\theta, \theta^{-1}]. \]
For branched double covers of knots, every $\spin^c$ structure $\mathfrak s$ has a unique associated real $\spin^c$ structure, so the right-hand side could equivalently be phrased with $\mathfrak s_r$ in place of $\mathfrak s$. In this note, we supply a proof of their conjecture, via showing the following.

\begin{thm} \label{thm:main} Let $(M, L_0, L_1)$ be a triple such that $M$ is an exact symplectic manifold which has the structure of a symplectization near infinity, and $L_0$ and $L_1$ are compact exact Lagrangians. Suppose that $R$ is an anti-symplectic involution on $M$ which interchanges $L_0$ and $L_1$. There is a spectral sequence whose $E_1$ page is
\[HF(M,L_0,L_1) \otimes \bF[\theta, \theta^{-1}]\]
and whose $E_{\infty}$ page is isomorphic to
\[ \HFR(M, L_0, L_1) \otimes \bF[\theta, \theta^{-1}].\]
\end{thm}
The strategy of the proof is to replace the anti-symplectic involution by a symplectic involution to which the localization theorem of Seidel-Smith and Large may be applied. Specifically, we replace $(M, \omega)$ with $(M \times M^-, \omega\oplus -\omega)$. Allowing $R$ to interchange the factors of the product gives a symplectic involution on $M \times M^-$. For our two Lagrangians we consider the product $L_0 \times L_1$ and the diagonal $\Delta$, each of which is preserved setwise by $S$. Standard Lagrangian Floer theory shows that \[\HF(M, L_0, L_1) \simeq \HF(M \times M^-, L_0\times L_1, \Delta).\] The fixed sets $(M^S, \Delta^S, (L_0 \times L_1)^S)$ are embedded copies of $(M, L_0, M^R)$, so that \[\HF(M^S, \Delta^S, (L_0 \times L_1)^S) \simeq \HFR(M, L_0, L_1).\] After checking the technical hypotheses of Large's version of the localization theorem the result follows. The details are carried out in Section~\ref{sec:proof}.

This theorem then has the following special case in Heegaard Floer theory. In the statement that follows, the map $\tau_*$ is the action on $\HFhat(Y)$ associated to the diffeomorphism $\tau$ on Heegaard Floer homology as studied by Juh{\'a}sz-Thurston-Zemke \cite{JTNaturality}. Recall that if $Y$ is a rational homology sphere, this map is well-defined; if $Y$ is not, it is necessary to fix a basepoint $w$, which we always choose to be on the fixed set of $\tau$. We include $w$ in the notation where its presence is relevant. The map $\iota$ is the conjugation involution on Heegaard Floer homology studied by the author and Manolescu \cite{HMInvolutive}. The constructions of these maps are reviewed in Section~\ref{sec:actions}.

\begin{thm} \label{thm:branched} Let $Y$ be a 3-manifold admitting an involution $\tau$ whose fixed set has codimension two together with $w$ a basepoint on the fixed set, and let $\mathfrak s \in \Spin^c(Y)$ be a $spin^c$ structure on $Y$. There is a spectral sequence whose $E_1$ page is isomorphic to 
\[ \HFhat(Y,w, \mathfrak s) \otimes \bF[\theta, \theta^{-1}] \]
and whose $E_{\infty}$ page is isomorphic to
\[ \bigoplus_{\mathfrak s_r \in \mathfrak s}\HFRhat(Y, \mathfrak s_r, \tau) \otimes \bF[\theta, \theta^{-1}].\]
Here the sum $\bigoplus_{\mathfrak s_r \in \mathfrak s}$ is taken over the set of real $\spin^c$ structures associated to $\mathfrak s$. There is therefore a dimension inequality
\[ \dim \HFhat(Y,w, \mathfrak s)  \geq \dim \bigoplus_{\mathfrak s_r \in \mathfrak s}\HFRhat(Y, \mathfrak s_r, \tau).\]
Furthermore if $(Y,\tau)$ admits a real Heegaard diagram $H = (\Sigma, \alphas, \betas, w)$ with the property that there is an $R$-symmetric family of almost complex structures $J$ on $\Sigma$ such that $\Sym^g(J)$ achieves transversality for $(\Sym^g(\Sigma -\{w\}), \T_{\alphas}, \T_{\betas})$, then the $d_1$ differential of this spectral sequence is
\[ (1+\iota_* \tau_*)\theta.\]
In particular, if $\Sigma(K)$ is the branched double cover of a knot $K$ in the three-sphere, let $\mathfrak s \in \Spin^c(\Sigma(K))$ be a $\spin^c$ structure on $\Sigma(K)$ and $\mathfrak s_r$ be its unique associated real $\spin^c$ structure. There is a spectral sequence whose $E_1$ page is isomorphic to
\[ \HFhat(\Sigma(K), \mathfrak s) \otimes \bF[\theta, \theta^{-1}]\]
and whose $E_{\infty}$ page is isomorphic to 
\[ \HFRhat(\Sigma(K), \mathfrak s_r, \tau) \otimes \bF[\theta, \theta^{-1}]. \]
\end{thm}

Let us remark on the condition associated to the computation of the $d_1$ differential. Guth and Manolescu show that not every real Heegaard diagram admits an $R$-symmetric family of almost complex structures achieving transversality \cite[Remark 2.4 and Example 4.7]{GM:real}. However, this does not rule out the possibility that every three-manifold possesses some real Heegaard diagram satisfying the condition above; for example, in the analogous case of symplectic actions, the arguments of \cite[Section 5.2]{HLS-equivariant} show this is satisfied by equivariant diagrams which are nice in the sense of Sarkar and Wang \cite{SW:nice}. We expect that in order to show the computation of the $d_1$ differential above holds for all three-manifolds, it suffices to prove that a suitable version of the Sarkar-Wang algorithm can be run symmetrically on a real Heegaard diagram, and then to use the resulting real nice diagrams to carry out an index computation of a similar flavor to \cite[Section 5.2]{HLS-equivariant}.

We now discuss how the spectral sequence for the branched double cover of Theorem~\ref{thm:branched} differs from \eqref{eqn:original} and \eqref{eqn:original-manifolds}. If $Y$ is a real 3-manifold, the induced isomorphism $\tau_*$ sends $\HFhat(Y, w, \mathfrak s) \rightarrow \HFhat (Y, w, \overline{\mathfrak s})$, where $\overline{\mathfrak s}$ is the conjugate $\spin^c$ structure to $\mathfrak s$. In the spectral sequence \eqref{eqn:original-manifolds}, the $d_1$ differential is $(1+\tau_*)\theta$, and similarly for \eqref{eqn:original} with the action induced by $\tau$ on the knot Floer homology. Therefore the Heegaard Floer homology in $\spin^c$ structures which are not conjugation invariant cannot survive to the $E_2$ page of either spectral sequence, and we obtain interesting spectral sequences from only the conjugation-invariant $\spin^c$ structures. For branched double covers of knots, there is only one conjugation-invariant $\spin^c$ structure, namely $\mathfrak s_0$, resulting in the form of the spectral sequence \eqref{eqn:original} seen.

This cancellation does not occur in the spectral sequence of Theorem~\ref{thm:branched}, where the $d_1$ differential fixes the $\spin^c$ structure. Indeed, this spectral sequence splits along $\spin^c$ structures and is interesting for any $\spin^c$ structure $\mathfrak s$. This tallies with the computation of the $d_1$ differential: the map $\iota$ also conjugates the $\spin^c$ structure, so that in total $\iota_*\tau_*$ preserves $\spin^c$ structures.

As Guth and Manolescu point out, Theorem~\ref{thm:branched} has the following corollary for branched double covers which are Heegaard Floer $L$-spaces. Recall that by taking the Euler characteristic of real Heegaard Floer homology in the branched double cover, they define a new ``Miyazawa-type'' knot invariant $\chi_{\mathfrak s}(K) = \chi(\HFRhat(\Sigma(K), \mathfrak s, \tau)).$

\begin{cor} \label{cor:lspace} Let $\Sigma(K)$ be an $L$-space. Then for every $\spin^c$ structure $\mathfrak s$ on $\Sigma(K)$, 
\[ \dim\left(\HFRhat(\Sigma(K), \mathfrak s, \tau)\right) = 1.\]
It follows that $\chi_{\mathfrak s}(K) = \pm 1$ for all $\mathfrak s \in \Spin^c(\Sigma(K))$. \end{cor}

The initial constructions of \cite{GM:real} are solely for involutions $\tau$ on $3$-manifolds $Y$. However, Theorem~\ref{thm:main} also has applications to the case of an action $\tau$ on $(Y,K)$ where $K$ is a nullhomologous knot in $Y$. As previously, one requires that the fixed set of $\tau$ on $Y$ have codimension two. We will consider the case of a  knot in the three-sphere which is \emph{strongly invertible}; that is, there is an involution on $(S^3,K)$ which is orientation-preserving on $S^3$ and orientation-reversing on $K$, so that the fixed set of the involution is an unknotted axis intersecting $K$ in two points. The knot therefore divides the axis into two halves. A choice of half-axis and an orientation for this half-axis is a \emph{direction} \cite[Section 1]{Sakumainvertible}. After making this choice we obtain a directed strongly invertible knot.

Recall that the knot Floer homology of a nullhomologous knot $K$ in a three manifold $Y$ is constructed from a Heegaard diagram $H = (\Sigma, \alphas, \betas, w, z)$ such that $(\Sigma, \alphas, \betas, w)$ is a Heegaard diagram for $Y$ and $z$ is placed such that connecting $w$ to $z$ in $U_{\alpha}$ and $z$ to $w$ in $U_{\beta}$ recovers the knot $K$. The hat variant of the knot Floer homology is then $\HFKhat(Y,K) = HF(\Sym^g(\Sigma - \{z,w\}), \T_{\alphas}, \T_{\betas})$. The result splits along an integral Alexander grading, or equivalently relative $\spin^c$ structure, so that $\HFKhat(S^3,K) \simeq \bigoplus_{i \in \mathbb Z} \HFKhat(S^3,K,i)$. 

Following the same recipe as above, suppose one has a Heegaard diagram $H = (\Sigma, \alphas, \betas, w, z)$ for a strongly invertible knot $K$ which is symmetric in the sense that  $\tau|_{\Sigma}=R$ is orientation-reversing, fixes $z$ and $w$, and exchanges $\alpha$ curves with $\beta$ curves. If $K$ is directed we additionally insist that $z$ be the starting point of the oriented half-axis and $w$ be the endpoint. One can then define an associated real knot Floer homology, to wit
\[ \HFKR(H) := \HFR( \Sym^g(\Sigma -\{w,z\}), \T_{\alphas}, \T_{\betas}).\]
The result splits as $\HFKR(H) \simeq \bigoplus_{i \in \mathbb Z} \HFKR(H,i)$ along Alexander gradings; we do not attempt to consider real Alexander gradings in the present work. One expects the arguments of \cite[Section 3]{GM:real} may be adapted to show that this theory is an invariant of the directed strongly invertible knot, although given the limited focus of this note we do not attempt a full argument here. However, in Section~\ref{sec:examples-knots} we construct doubly-pointed real Heegaard diagrams for directed strongly invertible knots $K$, and show the existence of the following spectral sequence.

\begin{prop}\label{prop:si} Let $K \subset S^3$ be a directed strongly invertible knot. There is a spectral sequence whose $E_1$ page is isomorphic to
\[ \HFKhat(S^3, K, i) \otimes \bF[\theta, \theta^{-1}] \]
and whose $E_{\infty}$ page is isomorphic to
\[ \HFKR (H,i) \otimes \bF[\theta, \theta^{-1}]\]
where $H$ is any real doubly-pointed Heegaard diagram for the directed strongly invertible knot. There is therefore a dimension inequality
\[\dim \HFKhat(S^3, K, i) \geq \HFKR (H,i).\]
Moreover if the directed strongly invertible knot admits a real Heegaard diagram $H = (\Sigma, \T_{\alphas}, \T_{\betas}, w,z)$ with the property that there is an $R$-symmetric family of almost complex structures $J$ on $\Sigma$ such that $\Sym^g(J)$ achieves transversality for the triple $(\Sym^g(\Sigma -\{w,z\}), \T_{\alphas}, \T_{\betas})$, then the $d_1$ differential of this spectral sequence is
\[ (1 + (\iota_K \tau_K)_*) \theta.\]
\end{prop}

Analogously to the previous statement, $\iota_K$ denotes the involutive knot Floer conjugation map of the author and Manolescu \cite{HMInvolutive}, and $\tau_K$ denotes the action of the knot symmetry as constructed by Mallick \cite{Mallick:surgery} and Dai, Mallick, and Stoffregen \cite{DMS:equivariant} on the knot Floer homology; we review these maps in Section~\ref{sec:actions}. The role of the direction in the statement is as follows. Each of the maps $\iota_K$ and $\tau_K$ is well-defined up to chain homotopy equivalence, but they are not simultaneously well-defined up to chain homotopy equivalence without a choice of direction \cite[Section 3.5]{DMS:equivariant}. In general, reversing the direction corresponds to postcomposing (``twisting'') one of $\tau_K$ or $\iota_K$ by the Sarkar basepoint-moving involution \cite{SarkarMovingBasepoints, ZemQuasi}; the two possible postcompositions produce chain homotopy equivalent triples \cite[Lemma 2.21]{DMS:equivariant}.

We observe some examples of the spectral sequence of Proposition~\ref{prop:si} for simple cases.

\begin{example} Recall that an $L$-space knot is a knot which admits a positive surgery to a Heegaard Floer $L$-space. If $K$ is an $L$-space knot, then $\dim(\HFKhat(S^3, K, i)) \leq 1$ for all $i$ \cite{OSlens}. Therefore if $K$ is a strongly invertible $L$-space knot, such as for example a torus knot with its unique strong inversion, the spectral sequence of Proposition~\ref{prop:si} collapses on the $E_1$ page. Indeed, if $H$ is a real Heegaard diagram for $K$, we conclude that in each Alexander grading $\HFKR(H, i) \simeq \HFKhat(S^3,K,i)$, ignoring the homological grading, for any real Heegaard diagram $H$ for $(S^3,K)$. 
\end{example}

\begin{example} Recall that a thin knot $K$ is one for which the difference between the Alexander grading and homological grading of $\HFKhat(S^3,K)$ is a constant; in particular, such that in each Alexander grading the knot Floer homology lies in a single homological grading. The set of thin knots includes the alternating knots. Since the spectral sequence of Proposition~\ref{prop:si} splits along Alexander gradings, it follows that for a strongly invertible thin knot, the only potentially nontrivial differential in the spectral sequence of Proposition~\ref{prop:si} is $d_1$. The spectral sequence therefore collapses on the $E_2$ page. \end{example}

At the close of Section~\ref{sec:examples-knots} we additionally consider some more speculative computations which assume that it is possible to realize the computation of the $d_1$ differential of Proposition~\ref{prop:si}.

\begin{rem} The construction of the spectral sequence of Proposition \ref{prop:si} uses Heegaard diagrams adapted from \emph{transvergent} diagrams for strongly invertible knots; in the language introduced by Boyle and Issa \cite{Boyle:alternating,BI:genera}, a transvergent diagram is one in which the symmetry axis lies in the plane of the page, as in Figure \ref{fig:transvergentdiagrams}. Recently, following the general strategy for Heegaard diagrams with orientation-preserving actions discussed above, Parikh \cite{Parikh} constructed a spectral sequence for strongly invertible knots using Heegaard diagrams adapted from \emph{intravergent} diagrams, which is to say diagrams in which the symmetry axis is perpendicular to the page, of the form:
\[ \HFKhat(S^3,K) \otimes \bF[\theta, \theta^{-1}] \rightrightarrows \HFhat(S^3) \otimes \bF[\theta, \theta^{-1}] \simeq \bF[\theta, \theta^{-1}].\]
Similarly to the case of \eqref{eqn:original} and \eqref{eqn:original-manifolds}, the differential on the $E_1$ page of Parikh's spectral sequence is $(1+(\tau_K)_*)\theta$; in particular, it reverses the sign of the Alexander grading, hence immediately cancelling $\HFKhat(S^3,K,i)$ and $\HFKhat(S^3,K,-i)$ for $i\neq 0$. Parikh's spectral sequence can therefore be taken to start at $\HFKhat(S^3,K,0)$. This is not true of the spectral sequence of Proposition~\ref{prop:si}, the differentials of which fix the Alexander grading.

Separately from the constructions of this paper, one may ask whether it is possible to define a localization spectral sequence from Heegaard diagrams adapted from transvergent diagrams which agrees with Parikh's as far as the $E_1$ page; that is, whose $d_1$ differential is $(1+(\tau_K)_*)\theta$. The reader should compare this to the analogous situation in the setting of Khovanov homology: it was recently shown by Chen and Yang \cite{CY:flip} that in the Khovanov setting the actions arising from transvergent and intravergent diagrams agree up to chain homotopy, leaving open the question of whether the spectral sequences induced by these actions agree. In the setting of Heegaard Floer homology, Juh{\'a}sz-Thurston-Zemke naturality \cite{JTNaturality} as applied by Mallick \cite[Section 6]{Mallick:surgery} implies the actions induced on $\widehat{HFK}$ by $\tau$ are the same for any Heegaard diagram, but this does not itself imply the remaining pages of localization spectral sequences defined from these two constructions of the action must agree. \end{rem}

\begin{rem} We expect that the spectral sequences of Theorem~\ref{thm:branched} and Proposition \ref{prop:si} may be shown to be invariants of the three-manifold and its involution, or of the directed strongly invertible knot, as appropriate, by combining the techniques of \cite{GM:real} with the arguments used in \cite[Section 5]{HLS-equivariant} to show invariance of other localization spectral sequences in Heegaard Floer theory.
\end{rem}

\subsection*{Organization} This paper is organized as follows. In Section~\ref{sec:realhf} we review Guth and Manolescu's construction of real Lagrangian Floer homology and real Heegaard Floer homology. We also review the construction of the involutive symmetry and actions induced by diffeomorphism in Heegaard Floer theory. In Section~\ref{sec:Large} we review the localization theorem of Seidel-Smith and Large for symplectic involutions in Lagrangian Floer theory. In Section~\ref{sec:proof} we present the proof of Theorem~\ref{thm:main}. In Section~\ref{sec:examples} we discuss applications of Theorem~\ref{thm:main} to Heegaard Floer theory, and in particular prove Theorems~\ref{thm:branched} and Proposition \ref{prop:si}. Finally, in Section~\ref{sec:since-then} we collect some developments postdating the first appearance of this note for the reader's convenience.

\subsection*{Acknowledgments} I am grateful to Fraser Binns, Gary Guth, Tye Lidman, Robert Lipshitz, Abhishek Mallick, Ciprian Manolescu, Matt Stoffregen, and Yonghan Xiao for helpful conversations. Portions of this work were carried out at the 2025 Georgia International Topology Conference in May 2025, at a Workshop on Floer Homotopy Theory at Institut Mittag-Leffler in June 2025, and at the conference Categorification in Low-Dimensional Topology at Ruhr University Bochum in July 2025; I am grateful to all three conferences for their hospitality. Finally, my thanks to the referee for a careful reading and helpful comments.

\section{Review of real Lagrangian Floer and Heegaard Floer homology} \label{sec:realhf}

In this section we briefly review some features of Guth and Manolescu's real Lagrangian Floer and Heegaard Floer theories, focusing on those necessary to this note. In the final subsection we additionally review the construction of certain maps in ordinary Heegaard Floer theory, which will helpful to computing the $d_1$ differentials in our spectral sequences.

\subsection{Real Lagrangian Floer theory} \label{subsec:realsetup}

Here we collect the hypotheses and definitions of the version of real Lagrangian Floer theory required for this note, most of which were already mentioned in the introduction. We assume the reader is familiar with standard Lagrangian Floer homology as in \cite{Floer88:LagrangianHF}. Let $(M, \omega)$ be an exact symplectic manifold which is equivalent to a symplectization near infinity; equivalently, we assume that there is an exhausting function $g \co M \rightarrow [0, \infty)$ and an almost complex structure $J$ compatible with $\omega$ such that $\omega = -d d^{\mathbb C} g$, and such that the critical points of $g$ are contained in the preimage of some compact interval. Let $L_0$ and $L_1$ be a pair of compact exact Lagrangians. Finally, suppose that $R \co M \rightarrow M$ is an anti-symplectic involution on $M$ exchanging $L_0$ and $L_1$, so that $L_1 = R(L_0)$ and vice versa. Let $M^R$ denote the set of fixed points of $M$ under the involution, which is a new Lagrangian submanifold of $M$.

\begin{define} \cite[Section 2]{GM:real} Let $(M, L_0, L_1)$ be as above and $R$ be an anti-symplectic involution. The real Lagrangian Floer homology of $(M, L_0, L_1)$ and $R$ is
\[ \HFR(M, L_0, L_1) = \HF(M, L_0, M^R)\]
where the right-hand side denotes the ordinary Lagrangian Floer homology of the triple.
\end{define}

Guth and Manolescu equivalently construct this theory as a count of disks invariant under the action of $R$, using the following special families of almost complex structures.

\begin{define} An almost complex structure $J$ compatible with $\omega$ is said to be \emph{$R$-symmetric} if
\[ J \circ R_* = -R_* \circ J.\]
A time-dependent family $J_t$ of almost complex structures compatible with $\omega$ is $R$-symmetric if 
\[
J_t \circ R_* = -R_* \circ J_{1-t}.
\]
\end{define}

We let $\mathcal J_R$ denote the space of time-dependent families of almost complex structures compatible with $\omega$ which are additionally $R$-symmetric. The space $\mathcal J_R$ is nonempty and contractible; as Guth and Manolescu point out, this follows from (for example) the proof of \cite[Proposition 1.1]{Welschinger:rational}.

\begin{rem} In \cite{GM:real} the authors also define their theory for monotone symplectic manifolds. As this note is focused on localization results which are only known in the exact case, we do not include this part of their definition here.\end{rem}

\subsection{Real Heegaard Floer theory}

We now turn our attention to real Heegaard Floer theory. We assume the reader is familiar with ordinary Heegaard Floer theory as in \cite{OSDisks, OSProperties}. Let $Y$ be a $3$-manifold with an orientation-preserving involution $\tau$ whose fixed set has codimension two, called a real three-manifold.

\begin{define} \cite[Definition 3.1]{GM:real} A (singly-pointed) \emph{real Heegaard diagram} for $(Y, \tau)$ is a pair $(H, R)$ such that $H$ is a Heegaard diagram \[H = (\Sigma, \{\alpha_1, \dots, \alpha_g\}, \{\beta_1, \dots, \beta_g\}, w)\] and
\begin{itemize}
\item $R$ is an orientation-reversing involution on $\Sigma$;
\item $R(\alpha_i) =\beta_i$;
\item $w$ lies on the fixed set $\Sigma^R$;
\item Attaching handlebodies $U_{\alpha}$ and $U_{\beta}$ along the sets of attaching curves $\alphas = \{\alpha_1, \dots, \alpha_g\}$ and $\betas = \{\beta_1, \dots, \beta_g\}$ and considering the induced involution recovers $(Y, \tau)$.
\end{itemize}
\end{define}
Guth and Manolescu show that any two real Heegaard diagrams for $(Y, \tau)$ are connected by a series of suitable real Heegaard moves \cite[Proposition 3.14]{GM:real}. (They also consider multi-pointed real Heegaard diagrams; see \cite[Section 3.3]{GM:real}.) Moreover, they show \cite[Section 3.8]{GM:real} one can choose such real Heegaard diagrams to be weakly admissible in the sense of \cite[Definition 4.10]{OSDisks}. Given this setup, the triple $(\Sym^g(\Sigma -\{w\}), \T_{\alphas}, \T_{\betas})$ satisfies the symplectic hypotheses of Section~\ref{subsec:realsetup}. One then has the following.
\begin{define} \cite[Section 3]{GM:real} The hat variant of the real Heegaard Floer of $(Y, \tau)$ is
\[\HFRhat(Y, \tau) := \HFR(\Sym^g(\Sigma - \{w\}), \T_{\alphas}, \T_{\betas}). \]
\end{define}
Guth and Manolescu show this theory is an invariant of $(Y, \tau)$ up to isomorphism \cite[Section 5]{GM:real}.

Before moving on, we note that one may also consider the following analogous definition for a strongly invertible knot $K$ in $S^3$. Let $\tau$ be the action on $(S^3,K)$.

\begin{define} A (doubly-pointed) \emph{real Heegaard diagram} for $(S^3, K, \tau)$ is a pair $(H, R)$ such that $H$ is a Heegaard diagram \[H = (\Sigma, \{\alpha_1, \dots, \alpha_g\}, \{\beta_1, \dots, \beta_g\}, w, z)\] and
\begin{itemize}
\item $R$ is an orientation-reversing involution on $\Sigma$;
\item $R(\alpha_i) =\beta_i$;
\item $w,z$ lie on the fixed set $\Sigma^R$;
\item Attaching handlebodies $U_{\alpha}$ and $U_{\beta}$ along the sets of attaching curves $\alphas = \{\alpha_1, \dots, \alpha_g\}$ and $\betas = \{\beta_1, \dots, \beta_g\}$ and considering the induced involution recovers $(Y, \tau)$.
\item Connecting $w$ to $z$ via an arc in $U_{\alpha}$ and $z$ to $w$ via the image of the same arc in $U_{\beta}$ under $\tau$ recovers $K$ with its symmetry.
\end{itemize}
If $K$ is directed, we say that $H$ is a Heegaard diagram for the directed strongly invertible knot if $z$ lies at the starting point of the choice of oriented half-axis and $w$ lies at its endpoint.
\end{define}

With this in mind, one may set
\[\HFKR(H) =  \HF(\Sym^g(\Sigma -\{w,z\}), \T_{\alphas}, \T_{\betas}). \]
One expects that an adaptation of Guth and Manolescu's proof of the invariance of real Heegaard Floer homology for three-manifolds shows that this theory is an invariant of the knot and the symmetry.

\subsubsection{Real Euler and $\spin^c$ structures} We now give an abbreviated review of real $\spin^c$ structures, focusing on those aspects necessary to our proofs. A full discussion appears in \cite[Section 3.5]{GM:real}. 

Following Turaev's interpretation of $\spin^c$ structures on three manifolds \cite{Turaev:torsion}, we recall that an \emph{Euler structure} is a choice of non-vanishing vector field $v$ on $Y$, up to the following equivalence relation: two nonvanishing vector fields $v_0$ and $v_1$ are homologous if $v_0$ is homotopic to $v_1$ outside of some ball $B^3$. The set $\mathrm{Vec}(Y)$ of Euler structures is an affine copy of $H^2(Y; \mathbb Z)$, and in bijection with the set of $\spin^c$ structures $\Spin^c(Y)$. One obtains $\overline{\mathfrak s}$ the \emph{conjugate $\spin^c$ structure} to $\mathfrak s$ by replacing $v$ with $-v$.

Given a Heegaard diagram $H$ for $Y$, Ozsv{\'a}th  and Szab{\'o} \cite[Section 2.6]{OSDisks} define a map 
\[ \T_{\alphas} \pitchfork \T_{\betas} \rightarrow \Spin^c(Y)\]
partitioning intersection points into $\spin^c$ structures as follows. One may choose a self-indexing Morse $f$ function associated to $H$, with a single index zero and index three critical point, and $g$ critical points of index one and two, so the surface $\Sigma$ is the level set $f^{-1}(3/2)$, the $\alpha$ curves are the intersections of the ascending manifolds of index one critical points with $\Sigma$, and the $\beta$ curves are the intersections of the descending manifolds of index two critical points with $\Sigma$. Given this setup, every intersection point ${\bf x}$ determines a $g$-tuple of flowlines from index two to index one critical points, which we consider along with a flowline from the index three critical point to the index zero critical point passing through $w$. On the complement of a neighborhood of these trajectories the gradient vector field $df$ of $f$ is nonvanishing; on a neighborhood of these trajectories it can be homotoped to be nonvanishing. Hence it determines an Euler structure, and therefore a $\spin^c$ structure $\mathfrak s_w({\bf x})$. Moreover, Ozsv{\'a}th and Szab{\'o} show \cite[Section 2.4]{OSDisks} that $\Spin^c(Y)$ is in bijection with the path space $\mathcal P(\T_{\alphas}, \T_{\betas})$ in $\Sym^g(\Sigma -\{w\})$, and two intersection points ${\bf x}$ and ${\bf y}$ are in the same $\spin^c$ structure if their corresponding constant paths lie in the same path component of the path space $\mathcal P(\T_{\alphas}, \T_{\betas})$.

Let us now review the analog of this story in the real case. Given a real three-manifold $(Y, \tau)$, Guth and Manolescu show that the fixed intersection points in $(\T_{\alphas} \pitchfork \T_{\betas})^R$ may be further partitioned into \emph{real $\spin^c$ structures}. For our purposes, it suffices to note that these also have a description in terms of vector fields on $Y$ up to an appropriate equivalence relation, as follows. A \emph{real vector field} on $Y$ is a nonvanishing vector field $v$ such that $\tau_* v = -v$. An Euler structure on $Y$ \emph{admits a real structure} if it can be represented by a real vector field. Two real vector fields $v_0$ and $v_1$ are real homologous under the following conditions. Firstly, $v_0$ and $v_1$ must be homologous as vector fields. Secondly, if $TY$ is given a $\tau$-invariant metric, we consider the complex linear bundles with complex anti-linear involutions $(\langle v_0 \rangle^{\perp}, \tau_*)$ and $(\langle v_1 \rangle^{\perp}, \tau_*)$. We require that these are equivalent as real complex line bundles, that is, there is an isomorphism between them which commutes with $\tau_*$. The set $\mathrm{RVec}(Y, \tau)$ of real Euler structures is then the set of real vector fields up to the equivalence relation of being real homologous, and is in bijection with the set of real $\spin^c$ structures on $\mathrm{R}\Spin^c(Y, \tau)$, which we will not need to otherwise define. Notice that this implies that a real $\spin^c$ structure always has an ordinary underlying $\spin^c$ structure, and any given ordinary $\spin^c$ structure has some number of associated real $\spin^c$ structures, including the possibility of zero.

We now consider how invariant intersection points determine real Euler structures. Given a real Heegaard diagram, the associated self-indexing Morse function $f \co Y \rightarrow [0,3]$ may be chosen such that $f \circ \tau = 3-f$, implying that $df \circ d\tau = -df$. In this case the index one and index two critical points are interchanged by $\tau$, and given an invariant intersection point ${\bf x}$, the collection of flowlines from index one to index two critical points defined by the elements of ${\bf x}$ is also invariant, with a negation. It follows that ${\bf x}$ specifies a real Euler structure, and thus a real $\spin^c$ structure, denoted $\mathfrak s^R_{w}({\bf x})$.

We make the following observation regarding the action of $R$ on the generators associated to a real Heegaard diagram $(H,R)$. 

\begin{lem} \label{lem:spincsame} Let $(H,R)$ be a real Heegaard diagram for a real 3-manifold $(Y,\tau)$. Then if ${\bf x}$ and $R({\bf x})$ are both regarded as generators in $\CFhat(H)$, we have $\mathfrak{s}_w(R({\bf x})) = \mathfrak{s}_w({\bf x})$. \end{lem}

\begin{proof} Consider a real Heegaard diagram $(H,R)$ for $(Y, \tau)$. If ${\bf x} \in \T_{\alphas} \pitchfork \T_{\betas}$, so that ${\bf x}$ is a generator in $\CFhat(H)$, then the construction of $f$ implies that $df \circ d\tau =-df$. Therefore if we regard $R({\bf x})$ as a generator of $\CFhat(R(H))$, we have that $\mathfrak s_{w}(R({\bf x})) = \overline{\mathfrak s_{w}(x)}$. However, if instead we regard $R({\bf x})$ as a generator of $\CFhat(H)$, then given a set of flowlines for $f$ specified by ${\bf x}$, applying $\tau$ and then reversing the sign produces a set of flowlines for $f$ specified by $R({\bf x})$. We conclude that, if both intersection points ${\bf x}$ and $R({\bf x})$ are regarded as lying in $\CFhat(H)$, we have $\mathfrak{s}_w(R({\bf x})) = \mathfrak{s}_w({\bf x})$. \end{proof}

Finally, we briefly discuss the case of knots. Recall that if $H = (\Sigma, \alphas, \betas, w,z)$ is a Heegaard diagram for a knot $K$ in the three-sphere, then there is an assignment of an Alexander grading $A({\bf x}) \in \mathbb Z$ to each intersection point ${\bf x} \in \T_{\alphas} \pitchfork \T_{\betas}$, and two intersection points have the same Alexander gradings if an only if the constant paths at those points lie in the same path component of $\mathcal P(\T_{\alphas}, \T_{\betas})$ in $\Sym^g(\Sigma - \{z,w\})$ \cite[Section 2]{OSKnots}. A full treatment of real knot Floer homology should also consider real Alexander gradings, but these will not be necessary for the purposes of this note. Instead we conclude with the analog of Lemma~\ref{lem:spincsame}, done slightly more concretely in the knot case.

\begin{lem} \label{lem:Alexandersame} If $(H,R)$ is a real Heegaard diagram for a strongly invertible knot $K$ and ${\bf x} \in \T_{\alphas} \pitchfork \T_{\betas}$, then if ${\bf x}$ and $R({\bf x})$ are both regarded as generators of $\CFKhat(H)$, they lie in the same Alexander grading. \end{lem}

\begin{proof} We first note that since $|\T_{\alphas} \pitchfork \T_{\betas}|$ is always odd for knots in $S^3$, and $R$ either fixes generators or exchanges them in pairs, there must be at least one fixed generator ${\bf y}$. Moreover, for any generator ${\bf x}$ there is a topological disk $\phi$ from ${\bf x}$ to ${\bf y}$ in $(\Sym^g(\Sigma), \T_{\alphas}, \T_{\betas})$. Then $A({\bf x}) - A({\bf y}) = n_{z}(\phi)-n_w(\phi)$, where $n_z(\phi) = \#(\mathrm{Im}(\phi) \cap V_z)$ is the intersection number between the image of $\phi$ and the divisor $V_z = \{z\} \times \Sym^{g-1}(\Sigma)$, and likewise $n_{w}$ \cite[Lemma 2.5]{OSKnots}. Since $R$ fixes $z$ and $w$, the image of $R(\phi)$ considered as a disk in $(\Sym^g(\Sigma), \T_{\alphas}, \T_{\betas})$ is a topological disk from $R({\bf x})$ to ${\bf y}$ with the same intersection numbers with the divisors. \end{proof}

\subsection{Actions on Heegaard Floer homology} \label{sec:actions}

We conclude with a discussion of certain actions on ordinary Heegaard Floer homology, which will allow us to compute the $d_1$ differentials of the spectral sequences of Theorem~\ref{thm:branched} and Proposition~\ref{prop:si} under special hypotheses.

Let $Y$ be a three-manifold, and $\mathfrak s$ be a $\spin^c$ structure on $Y$. Let $H = (\Sigma, \alphas, \betas, w)$ be a Heegaard diagram for $Y$. For this section, we must consider Heegaard diagrams together with a choice of a family of almost complex structures $J=J_t$ on $\Sigma$ such that $\Sym(J)$ achieves transversality for $(\Sym^g(\Sigma -\{w\}), \T_{\alphas}, \T_{\betas})$. Let $\cH=(H, J)$ refer to such a full set of Heegaard data.

We begin by reviewing the involutive conjugation symmetry for three-manifolds \cite{HMInvolutive}. Given $\cH$ a set of Heegaard data for $Y$ as above, we say that the conjugate diagram is $\overline{H} = (-\Sigma, \betas, \alphas, w)$, and the conjugate set of Heegaard data is $\overline{\cH} =(\overline{H}, \overline{J})$, where more precisely if $J=J_t$ then $\overline{J}_t = -J_{1-t}$. There is a chain isomorphism
\begin{align*} \eta \co \CFhat(\cH, \mathfrak s) &\rightarrow \CFhat(\overline{\cH}, \overline{\mathfrak s}) \\
							{\bf x} &\mapsto {\bf x}.
\end{align*}
Now, since $\cH$ and $\overline{\cH}$ are both sets of Heegaard data for $(Y,w)$, the naturality theorem for Heegaard Floer homology of Juh{\'a}sz-Thurston-Zemke \cite{JTNaturality} shows there is a naturality map, unique up to homotopy, sending
\[ \Phi_{\overline{\cH}, \cH} \co \CFhat(\overline{\cH}, \overline{\mathfrak s}) \rightarrow \CFhat(\cH, \overline{\mathfrak s}).\]
The composition $\Phi_{\overline{\cH}, \cH} \circ \eta$ is the involutive conjugation symmetry $\iota$ on $\CFhat(\cH)$ defined by the author and Manolescu. We let $\iota_*$ denote the induced map\footnote{In some references, the induced map $\iota_*$ on homology is called $J$ or $\mathcal J$; we avoid this out of fear of confusion with the almost complex structure.} on $\HFhat(Y,w, \mathfrak s)$.

We now review maps induced by diffeomorphisms. Given an orientation-preserving diffeomorphism $\tau$ on $Y$ fixing a basepoint $w$ and a set of Heegaard data $\cH = (H, J)$, we let $\tau(\cH) = (\tau H, \tau J)$, where $\tau H$ is the Heegaard diagram obtained by taking the image of $\Sigma$ together with its decorations under $\tau$ and $\tau J = d\tau \circ J \circ (d\tau)^{-1}$ is the pushforward of the almost complex structure. There is then a tautological chain isomorphism\footnote{Sometimes this chain isomorphism is called by the name of the diffeomorphism, in this case $\tau$.}
\begin{align*} \eta_{\tau} \co \CFhat(\cH, \mathfrak s) &\rightarrow \CFhat(\tau\cH, \tau_* \mathfrak s) \\
						{\bf x} &\mapsto \tau{\bf x}. \end{align*}
There is again a naturality map $\Phi_{\tau \cH, \cH} \co \CFhat(\tau\cH, \tau_* \mathfrak s) \rightarrow \CFhat(\cH, \tau_* \mathfrak s)$. The composition of these maps $\Phi_{\tau \cH, \cH} \circ \eta_{\tau}$ is the chain map $\tau$ on $\CFhat(\cH)$ which induces the map $\tau_*$ on $\HFhat(Y, w)$.

We observe the following special case. If $H$ is a real Heegaard diagram for $(Y,\tau)$ and $J$ is a family of $R$-symmetric almost complex structures on $\Sigma$ such that $\Sym^g(J)$ achieves transversality, then the conjugate set of Heegaard data $\overline{H}$ is exactly $\tau(\cH)$, and $\tau_* \mathfrak s = \overline{\mathfrak s}$. This implies the following.

\begin{lem} \label{lem:iotatau} Let $(Y, \tau)$ be a real 3-manifold and $\cH =(H,J)$ a set of Heegaard data such that $J$ is $R$-symmetric and $\Sym^g(J)$ achieves transversality. Then the composition $\iota \circ \tau$ is chain homotopic to the composition $\eta \circ \eta_{\tau}$. \end{lem}

\begin{proof} We may manipulate the composition map as follows:
\begin{align*}
\iota \circ \tau &\simeq \Phi_{\overline{\cH}, \cH} \circ \eta \circ  \Phi_{\tau(\mathcal H), \mathcal H} \circ \eta_{\tau} \\
				&\simeq \eta \circ \Phi_{\cH, \overline{\cH}} \circ  \Phi_{\overline{\mathcal H}, \mathcal H} \circ \eta_{\tau} \\
				&\simeq \eta \circ \eta_{\tau}.
\end{align*}
Here the second line uses the fact that $\eta \circ  \Phi_{\overline{\mathcal H}, \mathcal H} \simeq \eta \circ \Phi_{\cH, \overline{\cH}}$ \cite[Section 2]{HMInvolutive} and that $\tau(\cH) = \overline{\cH}$, and the third line uses the fact that naturality maps are unique up to chain homotopy \cite{JTNaturality}.
\end{proof}

We now turn our attention to the knot case, which is slightly more complicated. For an arbitrary knot $K$ in the three-sphere, let $H = (\Sigma, \alpha, \beta, w, z)$ be a Heegaard diagram for $K$, and $J= J_t$ any compatible family of almost complex structures on $\Sigma$ such that $\Sym^g(J)$ achieves transversality on $(\Sym^g(\Sigma - \{w,z\}), \T_{\alphas}, \T_{\betas})$. This gives us a set of Heegaard data $\cH = (H,J)$ for $(S^3, K)$.  The conjugate set of Heegaard data $\overline{\cH} = (\overline{H}, \overline{J})$ has $H = (-\Sigma, \betas, \alphas, z, w)$. There is a chain isomorphism
\[ \eta_K \co \CFKhat(\cH) \rightarrow \CFKhat(\overline{\cH}).\]
One cannot quite return to the original diagram via the naturality maps of Juh{\'a}sz-Thurston-Zemke, since the roles of the basepoints have been interchanged. Instead, one chooses a basepoint-moving automorphism $\rho$ of $(S^3, K)$ which switches $w$ and $z$ using a half-twist along $K$. Then there is a naturality map $\Phi_{\rho(\cH), \overline{\cH}}$. The involutive conjugation action $\iota_K$ is then the composition
\[ \CFKhat(\cH) \xrightarrow{\eta_{\rho}} \CFKhat(\rho \cH) \xrightarrow{\Phi_{\rho(\cH), \overline{\cH}}} \CFKhat(\overline{\cH}) \xrightarrow{\eta_K} \CFKhat(\cH)\]
where here $\eta_{\rho}$ is the tautological isomorphism between $\CFKhat(\cH)$ and $\CFKhat(\rho \cH)$ from applying $\rho$, defined analogously to $\eta_{\tau}$. The resulting map $\iota_K$ is order four up to homotopy \cite[Section 6.1]{HMInvolutive}.

Now suppose we have a directed strongly invertible knot $K$, with real Heegaard data $\cH=(H,J)$ consisting of a real Heegaard diagram $(H,R)$ for $K$ with its action and direction together with some family of almost complex structures achieving transversality. Following Mallick \cite[Definition 6.1]{Mallick:surgery}, the action $\tau_K$ associated to $K$ on $\CFKhat(\cH)$ may be defined as follows. Let $\tau(\cH)$ as previously denote the Heegaard data $\cH$ pushed forward by $\tau$. Furthermore let $\cH^r = (H^r, J)$ with $H^r = (\Sigma, \alpha, \beta, z, w)$, so that $\cH^r$ is a set of Heegaard data for the knot with the orientation reversed. There is a tautological automorphism between $\CFhat(\cH)$ and $\CFhat(\cH^r)$ by swapping the roles of the basepoints, called $sw$ in the literature.  Then $\tau_K$ is the composition of the maps
\[ \CFKhat(\cH) \xrightarrow{\eta_{\rho}} \CFKhat(\rho\cH) \xrightarrow{\Phi_{\rho \cH, \tau \cH^r }} \CFKhat(\tau \cH^r) \xrightarrow {\eta_{\tau}} \CFKhat(\cH^r) \xrightarrow{sw} \CFKhat(\cH). \]
Here $\eta_{\tau}$ is the chain isomorphism coming from applying $\tau$ to $\tau(\cH^r)$, bearing in mind that $\tau$ is topologically an involution. The resulting composition $\tau_K$ has $\tau_K^2$ chain homotopic to the identity \cite[Section 6]{Mallick:surgery}. 

We observe that in the special case that there is a family of $R$-symmetric almost complex structures achieving transversality, we have that $\tau(\cH^r)$ is precisely $\overline{\cH}$. Analogously to the 3-manifold case, we note the following property of the composition of these maps $\iota_K$ and $\tau_K$. 
\begin{lem} \label{lem:iotaKtauK} Let $H$ be a real Heegaard diagram for a strongly invertible knot $K$ admitting a compatible family of almost complex structures $J$ such that $J$ is $R$-symmetric and $\Sym^g(J)$ achieves transversality. Then the composition $\iota_K \circ \tau_K$ is chain homotopic to $\eta_{\tau} \circ \eta_K \circ sw$. \end{lem}
\begin{proof} We may manipulate the composition as follows.
\begin{align*} 
\iota_K \circ \tau_K &\simeq \iota_K \circ \tau_K^{-1} \\
						&\simeq \eta_K \circ \Phi_{\rho(\cH), \overline{\cH}} \circ \eta_{\rho} \circ \eta_{\rho}^{-1} \circ \Phi_{\tau(\cH^r), \rho(\cH)} \circ \eta_{\tau} \circ sw \\
						&\simeq \eta_K \circ \Phi_{\rho(\cH), \overline{\cH}} \circ \Phi_{\tau(\cH^r), \rho(\cH)} \circ \eta_{\tau} \circ sw \\
						&\simeq \eta_K \circ \eta_{\tau} \circ sw.
\end{align*}
The above rearrangement uses the following: firstly, $\eta_{\tau}$ and $sw$ are their own inverses. Secondly, up to homotopy the inverse of $\Phi_{\rho(\cH), \tau(\cH^r)}$ is the naturality map $\Phi_{\tau(\cH^r), \rho(\cH)}$. Finally, as noted above, $\tau(\cH^r)$ is $\overline{\cH}$. Note that it is also the case that $sw$ commutes with both $\eta_{\tau}$ and $\eta_K$ and in particular can be placed at any location in the composition above.
\end{proof}

\begin{rem} The expert reader may be aware that it is more common to define $\tau_K$ using Heegaard data for which the basepoints are interchanged rather than fixed, as in \cite{DMS:equivariant} and the initial definitions of \cite{Mallick:surgery}. We briefly discuss the correspondence between these pictures, which also helps to explain the role of the direction in our discussion. Given a strongly invertible knot $K$, the fixed axis divides $K$ into two subarcs. Given a decoration on $K$, Dai-Mallick-Stoffregen place basepoints $z'$ and $w'$ on $K$ such that they determine an orientation on $K$ in which the subarc containing $z'$ has orientation matching that of the oriented half-axis \cite[Figure 9]{DMS:equivariant}; such a choice of orientation and basepoint placement is said to be \emph{compatible with the direction}. With this choice, they show that $(\CFK^-(S^3,K), \tau_K, \iota_K)$ is determined up to chain homotopy equivalence. In this work, we must work with fixed basepoints; we place $z$ at the starting point of the oriented axis and $w$ at the endpoint, so that pushing $z$ and $w$ by a quarter-circle in the direction of the orientation produces a $z'$ and $w'$ compatible with the direction in the sense of Dai-Mallick-Stoffregen. Mallick \cite[Proposition 6.2]{Mallick:surgery} shows that the map associated to this quarter Dehn twist interchanges $\tau_K$ defined by the formula above with $\tau_K$ as defined for the nonfixed basepoints $z'$ and $w'$, up to homotopy. The same map also interchanges $\iota_K$ defined for on $\CFK(S^3,K,z,w)$ and $\CFK(S^3,K,z',w')$  up to homotopy \cite[Proposition 2.8 and Proposition 6.3]{HMInvolutive}, so our choice of the placement of the $z$ and $w$ basepoints is also compatible with the direction. One may straightforwardly check that rotating by a quarter circle against the direction of the orientation instead produces the twist by the Sarkar map associated to changing the direction in \cite[Section 3.5]{DMS:equivariant}.

Finally, note that it is previously known that $\iota_K \tau_K$ is an involution up to homotopy, even though $\iota_K$ is only order four up to homotopy. Indeed, Dai, Mallick, and Stoffregen \cite[Theorem 1.7]{DMS:equivariant} show that $\tau_K \circ \iota_K \simeq \xi_K \circ \iota_K \circ \tau_K$, where $\xi_K$ is Sarkar's involution associated to a Dehn twist around $K$. As $\xi_K \iota_K \simeq \iota_K^{-1}$, it follows that $(\iota_K \tau_K)^2 \simeq \mathrm{Id}$. 
\end{rem}

\section{Review of localization in Lagrangian Floer theory} \label{sec:Large}

In this section we recall Large's generalization of Seidel and Smith's localization theorem for Lagrangian Floer theory. We begin by reviewing the following definitions, adapted from \cite[Section 3.2]{Large}.

\begin{define} Let $M$ be a symplectic manifold containing Lagrangians $L_0$ and $L_1$. A \emph{set of polarization data} for $(M, L_0, L_1)$ is a triple $\mathfrak p = (E, F_0, F_1)$ such that
\begin{itemize}
\item $E$ is a symplectic vector bundle over $M$, and
\item $F_i$ is a Lagrangian subbundle of $E|_{L_i}$ for $i=0,1$.
\end{itemize}
\end{define}

Given $\mathfrak p = (E, F_0, F_1)$ a set of polarization data for $(M, L_0, L_1)$, one may stabilize by a trivial bundle to obtain $\mathfrak p \oplus \underline{\bC}^N = (E \oplus  \underline{\bC}^N, F_0 \oplus \underline{\bR}^N, F_1 \oplus i \underline{\bR}^N)$.

\begin{define} Let $\mathfrak p = (E, F_0, F_1)$ and $\mathfrak p' = (E', F_0', F_1')$ be two sets of polarization data for $(M, L_0, L_1)$. An \emph{isomorphism of polarization data} between $\mathfrak p$ and $\mathfrak p'$ is an isomorphism of symplectic vector bundles 
\[\psi \co E \rightarrow E'\]
such that there are homotopies through Lagrangian subbundles of $(E')|_{L_i}$ between $\psi(L_i)$ and $L_i'$ for $i=0,1$. A \emph{stable isomorphism of polarization data} between $\mathfrak p$ and $\mathfrak p'$ is an isomorphism of polarization data between $\mathfrak p \oplus \underline{\bC}^{N_1}$ and $\mathfrak p' \oplus \underline{\bC}^{N_2}$ for some $N_1$ and $N_2$.\end{define}

Let $(M,\omega)$ be an exact symplectic manifold which is convex at infinity, and let $L_0$ and $L_1$ be two exact Lagrangians such that for each $i=0,1$, either $L_i$ is compact, or it is the case that $M$ is a symplectization near infinity and $L_i$ is conical. (A Lagrangian is conical if near infinity it is the cone on a Legendrian in a level set of the exhausting function on $M$.) If both Lagrangians are conical near infinity we additionally require they be disjoint near infinity. Let $S$ be a symplectic involution on $M$, so that $S^*\omega = \omega$ and $S^2= \mathrm{Id}$, and let $S$ fix each of the two Lagrangians setwise, so that $S(L_i) = L_i$ for $i=1,2$.\footnote{More typically the involution is denoted $\tau$ or $\iota$; in this paper, $\tau$ is used for actions on $3$-manifolds, $\iota$ is used for the involutive conjugation involution, $S$ is used for symplectic actions on symplectic manifolds, and $R$ is used for anti-symplectic actions on symplectic manifolds.} Let $(M^S, L_0^S, L_1^S)$ denote the fixed sets with respect to $S$. In the case that either $M^S$ is connected or that all of the connected components of $M^S$ have the same dimension, it is straightforward to see that $M^S$ is itself an exact symplectic manifold which is convex at infinity and $L_0^S$ and $L_1^S$ are two exact Lagrangians satisfying the same hypotheses as previously, that is, each is compact or conical at infinity. The \emph{normal polarization} is the set of polarization data $(NM^S, NL_0^S, NL_1^S)$ consisting of the normal bundle to $M^S$ inside $M$, which is a symplectic vector bundle, together with the normal bundles $NL_i^S$ to $L_i^S$ inside $L_i$ for $i=0,1$, which are Lagrangian subbundles of $(NM^S)|_{L_i}$. This brings us to the following definition of Large \cite{Large}.

\begin{define} \label{def:stable-normal} A \emph{stable tangent-normal isomorphism} is a stable isomorphism of polarization data between the tangent polarization $(TM^{S}, TL_0^{S}, NL_1^{S})$ and the normal polarization $(NM^{S}, NL_0^{S}, NL_1^{S})$. \end{define}

We remind the reader that, since the symplectic group deformation retracts onto the unitary group, the theory of symplectic vector bundles is isomorphic to the theory of complex vector bundles. Therefore, the definitions above may be equivalently rephrased in terms of complex bundles with half-dimensional totally real subbundles.

We are now ready to state Large's localization theorem, as follows.

\begin{thm}\cite{Large}\label{thm:localization} Suppose that $(M,L_0,L_1)$ and $S$ satisfy the hypotheses above, and furthermore the triple $(M^{S}, L_0^{S}, L_1^{S})$ admits a stable tangent-normal isomorphism. There is a spectral sequence whose $E_1$ page is isomorphic to \[ HF(M, L_0,L_1)\otimes \mathbb F[\theta, \theta^{-1}]\] and whose $E_{\infty}$ page is isomorphic to \[ HF(M^S, L_0^S, L_1^S)\otimes \mathbb F[\theta, \theta^{-1}],\] and a corresponding dimension inequality
\[ \dim \left( HF(M, L_0, L_1)\right) \geq \dim \left( HF\left(M^S, L_0^S, L_1^S\right)\right).\]
\end{thm}

This generalizes Seidel and Smith's previous work in the case that the normal polarization is stably isomorphic to the trivial polarization \cite[Theorem 1, Theorem 20]{SS10}. These theorems have been applied broadly in Heegaard Floer theory \cite{Hendricks12:dcov-localization, Hendricks:periodic-localization, HLS-equivariant, Boyle, Large, HLL:branched, Parikh}, in Seidel and Smith's symplectic Khovanov homology \cite{SS10, HMR:annular}, and to study powers of symplectomorphisms \cite{Hendricks:symplecto}. 

We make two further remarks concerning the structure of the spectral sequence. First, it follows directly from the construction that as long as the topological map $S$ preserves the components of the path space $\mathcal P(L_0, L_1)$, the spectral sequence splits along the set of these path components; an example of splitting Large's construction along such path components is carried out in \cite[Section 2.2]{HLL:branched}.  

Second, the construction of the spectral sequence goes by constructing an equivariant model $C(M, L_0, L_1)$ for the Lagrangian Floer homology chain complex, via one of several possible equivalent methods \cite{SS10, HLS-equivariant, Large}. This complex has a grading-preserving chain map $\widetilde{S}_\#$, induced by $S$, which may be taken to be a chain involution, such that the spectral sequence is generated by the double complex $(C(M,L_0,L_1)\otimes \bF[\theta,\theta^{-1}], d +(1+\widetilde{S}_{\#})\theta)$. In particular, this model is in general different from the chain complex $\CF(M,L_0,L_1)$ computed in the standard way and generated by intersection points of $L_0$ and $L_1$. Indeed, assuming that $L_0$ and $L_1$ intersect transversely, the action $S$ on the original set of generators $L_0 \pitchfork L_1$ need not induce a chain map on $CF(M,L_0,L_1)$, and if by coincidence it induces a chain map need not induce one which is equvariantly chain homotopic to $\widetilde{S}_{\#}$. However, work of the author with Lipshitz and Sarkar \cite[Proposition 4.3]{HLS-equivariant} shows that if one has any equivariant family of almost complex structures achieving transversality on the original complex $CF(M, L_0, L_1)$, then one may use the chain involution $S_{\#}$ computed with respect to this family to induce the spectral sequence of Seidel-Smith and Large using the double complex $(CF(M,L_0,L_1)\otimes \bF[\theta,\theta^{-1}], d +(1+S_{\#})\theta)$. In particular the $d_1$ differential is then $(1+S_*)\theta$. This is the observation that will allow us to compute $d_1$ for the spectral sequences of Section~\ref{sec:examples} assuming the existence of Heegaard diagrams admitting such families of almost complex structures on the symmetric product.

\section{Proof of localization for real Lagrangian Floer homology} \label{sec:proof}

In this section we prove Theorem~\ref{thm:main}. In accord with the hypotheses, let $(M, L_0, L_1)$ be a triple such that $M$ is exact and has the structure of a symplectization near infinity, and $L_0$ and $L_1$ are compact exact Lagrangians. As previously, suppose that there is an anti-symplectic involution $R$ on $M$ interchanging $L_0$ and $L_1$. 

\begin{proof}[Proof of Theorem~\ref{thm:main}] As explained in the introduction, our strategy is to replace the anti-symplectic involution $R$ with a symplectic involution $S$ on a product symplectic manifold. In particular, consider the symplectic manifold $M \times M^-$ with symplectic form $\tilde{\omega}= \omega \oplus -\omega$. This manifold contains Lagrangian submanifolds $\cL = L_0 \times L_1$ and $\Delta = \{(x,x) : x \in M\}$. We note that $M \times M^-$ is still exact, since if $\omega = d\lambda$ then $\tilde{\omega}= d\left(\lambda \oplus -\lambda\right)$. Similarly, both $\cL$ and $\Delta$ are exact Lagrangians. We see that $\cL$ is compact. As for $\Delta$, one may check that if $M$ has the structure of a symplectization near infinity then the same is true of $M\times M^-$, and the diagonal is a conical Lagrangian; the details are carried out in \cite[Section 2.2]{Hendricks:symplecto}. It is additionally well-known that 
\[ \HF(M \times M^-, \cL, \Delta) \simeq HF(M, L_0, L_1)\]
via folding strips; c.f., e.g., \cite[Proposition 8.2]{Ganatra} or \cite[Lemma 2.9]{Hendricks:symplecto} for a review.

We now consider an involution on this product, as follows.
\begin{align*} 
S: M\times M^{-} \rightarrow M \times M^{-} \\
		(x,y) \mapsto (R(y), R(x) ) 
\end{align*}
We first observe that $S^*(\tilde{\omega}) = R^*(-\omega)\oplus R^*\omega = \omega \oplus -\omega = \tilde{\omega}$, so $S$ is indeed symplectic; furthermore since $R^2 = \mathrm{Id}$, we see that $S$ is also an involution. Moreover with respect to $S$ we have fixed sets
\begin{align*}
(M\times M^-)^{S} &= \{(x,R(x)): x \in M\} \\
\Delta^{S} &= \{(x,x): x = R(x) \} \\
\cL^{S} &= \{(x,y): x \in L_0, y \in L_1, y = R(x)\}.
\end{align*}
We see that if we consider the symplectic embedding
\begin{align*}
\phi \co (M, \omega) &\rightarrow (M \times M^{-}, \tilde{\omega}) \\
		x &\mapsto (x, R(x))
\end{align*}
we find that $\phi(M) = (M\times M^-)^{S}$ up to a factor of $1/2$ on the symplectic form, and furthermore $\phi(L_0) = \cL^{S}$ and $\phi(M^R) = \Delta^{S}$. We conclude that
\[ HF((M\times M^-)^{S}, \cL^{S}, \Delta^{S}) \simeq HFR(M, L_0, L_1).\]

It remains to check that $(M \times M^-, \cL, \Delta)$ admits a stable tangent normal isomorphism, after which we may apply Theorem~\ref{thm:localization} to obtain the spectral sequence of Theorem~\ref{thm:main}. We observe that the tangent bundle to $(M \times M^-)^{S}$ inside $M \times M^-$ is 
\begin{equation} \label{eqn:tangent} T_{(x, R(x))}((M\times M^-)^{S}) = \{ (v, R_*v) : v \in T_xM \}.\end{equation}
To parametrize the normal bundle, we choose an almost complex structure $J$ on $M$ which is compatible with $\omega$, that is, so that $\omega(Jv,Jw) = \omega(v,w)$ and $\omega(v,Jv) >0$ for all nonzero $v$. We further insist that $J$ be $R$-symmetric in the sense that $J R_* = -R_*J$, which as discussed in Section~\ref{subsec:realsetup} is always possible to arrange. (Note that this almost complex structure need not achieve transversality, since it is being used purely to establish the vector bundle isomorphism of Theorem \ref{thm:localization}.) We then have that $\tilde{J}=J \oplus -J$ is an almost complex structure on $M \times M^-$ compatible with $\tilde{\omega}$ with the property that $J S_* = S_* J$. We may write the normal bundle to $(M \times M^-)^{S}$ in $M \times M^-$ in the form
\begin{equation} \label{eqn:normal} N_{(x, R(x))}((M\times M^-)^{S}) = \{ (Jv, JR_*v) : v \in T_xM \}.\end{equation}
We pause to confirm that this is correct. It suffices to check that every vector in \eqref{eqn:normal} is orthogonal to every vector in \eqref{eqn:tangent}, since they are each clearly half-dimensional vector bundles. Recall that the metric on $M \times M^-$ is determined by the symplectic form and choice of compatible almost complex structure, and is given by
\[ g_{\tilde{J}} \langle(v_1,v_2), (w_1,w_2)\rangle := \tilde{\omega}((v_1,v_2), \tilde{J}(w_1,w_2)). \]
We therefore see that given $v,w \in T_xM$, we have that
\begin{align*}
g_{\tilde{J}} \langle (v, R_*v), (Jw, JR_*w) \rangle &= \tilde{\omega} ((v, R_* v), \tilde{J}(Jw, JR_*w)) \\
														&= \tilde{\omega}( (v, R_*v), (J^2 w, -J^2 R_*w)) \\
														&= \tilde{\omega} ( (v, R_* v), (-w, R_*w) ) \\
														&= \omega (v,-w) - \omega (R_* v, R_*w) \\
														&= -\omega(v,w) + \omega(v,w) \\
														&=0.
\end{align*}
We now describe a stable tangent-normal isomorphism between the tangent and normal polarizations of $((M\times M^-)^{S}, \Delta^{S}, \cL^{S})$. Our starting map is
\begin{align*}
\psi: T((M \times M^-)^{S}) &\rightarrow N( (M \times M^-)^{S}) \\
		(v,R_*v) &\mapsto (Jv, JR_*v)
\end{align*} 
We begin by checking that this is indeed an isomorphism of complex vector bundles; in particular, we see that
\[ \psi(\tilde{J}(v,R_*v)) = \psi((Jv, -JR_*v)) = (J^2v, -J^2R_*v) = \widetilde{J}(Jv,JR_*v) = \widetilde{J}\psi(v,R_*v).\]
We now consider the effects of the map $\psi$ on the tangent and normal bundles of $\Delta^{S}$ inside $\Delta$ and $\cL^{S}$ inside of $\cL$. We begin with the bundles associated with $\Delta^{S}$. Observe that a point $(x,x) \in \Delta^{S}$, such that $x=R(x)$, we have the following:
\begin{align*}
T_{(x,x)}\Delta &=\{(v,v) : v \in T_xM \} \\
T_{(x,x)} \Delta^{S} &= \{(v,v) : R_*v =v \} \\
N_{(x,x)} \Delta^{S} &= \{(Jw, Jw): R_*w =w \}.
\end{align*}
We observe that $\psi(T_{(x,x)} \Delta^{S}) = N_{(x,x)} \Delta^{S}$ on the nose, with no homotopy of Lagrangian bundles required. Now we turn our attention to the bundles associated to $\cL^{S}$. We see that at a point $(x,R(x)) \in \cL^{S}$, we have
\begin{align*}
T_{(x,R(x))} \cL &= \{(v,w): v \in T_xL_0, w \in T_{R(x)}L_1 \} \\
T_{(x,R(x))} \cL^{S} &= \{(v,R_*v):  v \in T_x L_0 \} \\
N_{(x, R(x))} \cL^{S} &= \{(w, -R_*w) : w \in T_x L_0 \}
\end{align*}

We now consider the effect of $\psi$ on $T \cL^S$. We observe that $\psi(T \cL^{S})$ is the totally real bundle whose fiber over each $(x, R(x))$ in $\cL^{S}$ is 
\[\{(Jv, JR_*v) : v \in T_x L_0 \}. \]
This image is exactly $\widetilde{J}(N \cL^{S})$. As it is always the case that a Lagrangian subbundle $F$ of a vector bundle $E$ is homotopic to $J(F)$ through rotation through Lagrangian subbundles, we are done.
\end{proof}

\section{Examples in Heegaard Floer theory} \label{sec:examples}

In this section we complete the proofs of Theorems \ref{thm:branched} and Proposition \ref{prop:si}, and make some remarks on the resulting spectral sequences.

\subsection{3-manifolds, including branched covers}

We now discuss the proof of Theorem~\ref{thm:branched}. 

\begin{proof}[Proof of Theorem \ref{thm:branched}] Let $(Y, \tau)$ be a real 3-manifold. By  \cite[Proposition 3.2]{GM:real}, which uses \cite[Proposition 2.4]{Nagase}, there is a real Heegaard diagram $H= (\Sigma, \alphas, \betas, z)$ representing $Y$. This Heegaard diagram may be taken to be weakly admissible, under which circumstance $\Sym^g(\Sigma - \{z\})$ admits an exact symplectic form $\omega$ with respect to which the two Lagrangian tori $\T_{\alphas}$ and $\T_{\betas}$ are exact; moreover, $\Sym^g(\Sigma - \{z\})$ admits an exhausting function $\phi$ whose critical points are contained in a compact set, giving the manifold the structure of a symplectization near infinity. (For a discussion of the symplectic form, see \cite{Perutz:Hamiltonian}; for details on how weak admissibility implies exactness of the Lagrangians and the structure of the manifold near infinity, see \cite[Proof of Proposition 4.2]{HLL:branched}.) In particular, this situation satisfies all of the hypotheses of Theorem~\ref{thm:main}, immediately giving us a spectral sequence
\begin{equation} \label{eqn:branchednospinc} \HFhat(Y) \otimes \bF[\theta, \theta^{-1}] \rightrightarrows \HFRhat(Y,\tau) \otimes \bF[\theta, \theta^{-1}].\end{equation}
We now remark on the splitting of this spectral sequence into $\spin^c$ structures.  Recall that our involution $S$ is constructed by considering the product 
\[\Sym^g(\Sigma-\{z\}) \times \Sym^g(\Sigma-\{z\})^- = \Sym^g(\Sigma - \{z\}) \times \Sym^g(-\Sigma -\{z\}).\]
Recall that if ${\bf x} \in \T_{\alphas} \pitchfork \T_{\betas}$, then $({\bf x},{\bf x})$ is a generator of the chain complex in the product, and the involution $S$ sends $({\bf x}, {\bf x}) \mapsto (R({\bf x}), R({\bf x}))$. In light of the identification between the complexes 
\[\CF(\Sym^g(\Sigma-\{z\}), \T_{\alphas}, \T_{\betas}) \simeq \CF(\Sym^g(\Sigma - \{z\}) \times \Sym^g(-\Sigma -\{z\}), \T_{\alphas} \times \T_{\betas}, \Delta)\] 
via folding strips, this implies that the set map induced by $S$ on the generators of $\CF(\Sym^g(\Sigma-\{z\}), \T_{\alphas}, \T_{\betas})$ sends ${\bf x}$ to $R({\bf x})$. Per Lemma~\ref{lem:spincsame}, these generators are in the same $\spin^c$ structure. We emphasize that this set map on the generators is not necessarily a chain map at this point; however, it nevertheless preserves the path components of $\mathcal P(L_0, L_1)$, implying that the spectral sequence splits along $\spin^c$ structures.

We now compute the $d_1$ differential in the spectral sequence in the special case that there is a Heegaard diagram $H$ which has an $R$-symmetric almost complex structure $J$ such that $\Sym^g(J)$ achieves transversality; that is, such that we have a set of Heegaard data $\cH$ suitable for defining Heegaard Floer homology such that $\overline{\cH}$ the conjugate set of Heegaard data is equal to $\tau (\cH)$ the pushfoward along $\tau$. In this situation the action of $S$ on generators is in fact a chain map. Again using the identification between the chain complex in the product and $\CFhat(\cH)$, the action is the chain map $\eta \circ \eta_{\tau}$ which sends ${\bf x} \mapsto R({\bf x})$ in $\CFhat(\cH)$. The discussion at the end of Section~\ref{sec:Large} now implies that the $d_1$ differential of the spectral sequence above is $(1+(\iota\tau)_*)\theta$. \end{proof}

We immediately also have a proof of Corollary~\ref{cor:lspace}.

\begin{proof}[Proof of Corollary~\ref{cor:lspace}] Let $\Sigma(K)$ be an $L$-space. Then $\dim(\HFhat(\Sigma(K), \mathfrak s)) = 1$ for all $\mathfrak s \in \Spin^c(Y)$. The dimension inequality of Theorem~\ref{thm:main} is in this case $\dim \HFhat(\Sigma(K), \mathfrak s) \geq \dim \HFRhat(\Sigma(K), \tau, \mathfrak s)$, from which we obtain the statement.
\end{proof}

\subsection{Strongly invertible knots} \label{sec:examples-knots}

Finally, we discuss the application to knot symmetry. Let $K$ be a strongly-invertible knot in $S^3$ and let the symmetry on $(S^3,K)$ be $\tau$. Given a transvergent diagram for $K$, we may construct a real Heegaard diagram as in Figure \ref{fig:transvergentdiagrams}. On the left, we see a relatively intuitive multi-pointed diagram on the two-sphere: we give the knot the structure of an equivariant bridge diagram with an odd number of bridges, label the endpoints of the bridges with basepoints according to the orientation of the knot, and circle all but one overcrossing bridge with $\beta$ curves and all but one undercrossing bridge with $\alpha$ curves. The involution is a reflection fixing the centerline of the diagram together with the point at infinity. It thus reverses the orientation of the Heegaard diagram and exchanges $\alpha$ curves with $\beta$ curves and vice versa. However, this diagram has basepoints off the fixed set; in order to apply Guth and Manolescu's construction to it, we wish to eliminate them. We may do this by adding handles connecting what were previously a pair of adjacent basepoints, and replacing the curves previously encircling the arc between that pair of basepoints with a curve running over the handle, obtaining a Heegaard diagram such as the one on the right-hand side of Figure \ref{fig:transvergentdiagrams}. If $K$ is directed we place the basepoints $z$ and $w$ so that $z$ lies at the starting point of the oriented half-axis and $w$ lies at the endpoint. The reader is invited to compare these to the intravergent-type diagrams of \cite[Figure 4]{Parikh}, which also incorporate a choice of direction, and carry an orientation-preserving action.

\begin{figure}
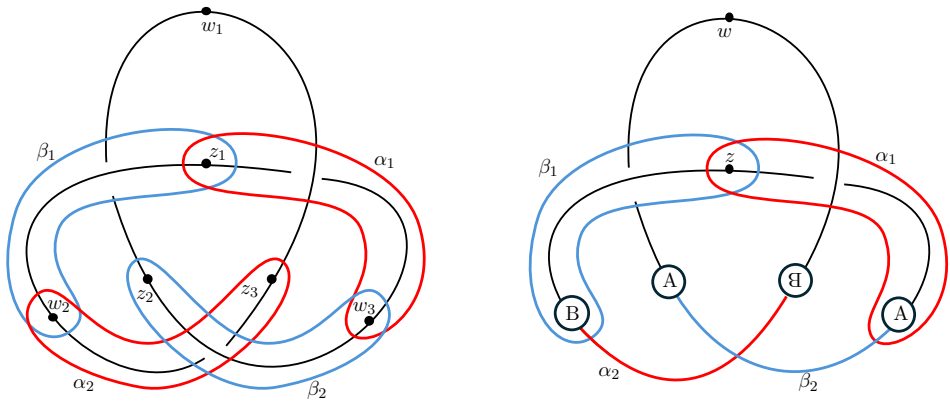
\caption{Left: A multi-pointed real Heegaard diagram for the right-handed trefoil with its unique strong inversion. Right: A related doubly-pointed real Heegaard diagram for the same knot and action.}
\label{fig:transvergentdiagrams}
\end{figure}

We are now ready to discuss the proof of Proposition~\ref{prop:si}.

\begin{proof} [Proof of Proposition~\ref{prop:si}] Let $H$ be a real Heegaard diagram for the directed strongly invertible knot $K$, as on the right side of Figure~\ref{fig:transvergentdiagrams}. As in the proof of Theorem~\ref{thm:branched}, it is straightforward to show that the triple $(\Sym^g(\Sigma -\{z,w\}), \T_{\alphas}, \T_{\betas})$ satisfies the symplectic requirements of Theorem~\ref{thm:main}. This produces the claimed spectral sequence; since by Lemma~\ref{lem:Alexandersame} the action $S$ respects the path components of $\mathcal P(L_0, L_1)$ it splits along Alexander gradings. An argument again nearly identical to the preceding case confirms that in the presence of a family of $R$-symmetric almost complex structures achieving transversality, the action induced by the map $S$ used in constructing the spectral sequence is the composition $\eta_K \circ \eta_{\tau} \circ sw$; by Lemma~\ref{lem:iotaKtauK} in homology this becomes $(\iota_K \tau_K)_*$. The spectral sequence therefore has $d_1 = (1+(\iota_K \tau_K)_*)\theta$ in this case.
\end{proof}

\begin{example} We conclude with two speculative computations, both carried out under the conjectural assumption that the $d_1$ differential of the spectral sequence of Proposition~\ref{prop:si} is indeed $(1+(\iota_K\tau_K)_*)\theta$. We begin with the figure-eight knot, which is thin. The knot Floer homology is five dimensional. Using one typical set of names for the generators (see, e.g. \cite[Figure 5]{DMS:equivariant}), we have a chain complex with trivial differential \[\CFKhat(S^3,4_1,1) \simeq \mathbb F\langle b \rangle, \qquad \CFKhat(S^3,4_1,0) \simeq \mathbb F_2\langle a,x,e \rangle, \qquad  \CFKhat(S^3,4_1, -1) = \mathbb F \langle c\rangle.\]
The map $\iota_K$ was computed for thin knots in \cite[Section 8]{HMInvolutive}. On $\CFKhat(S^3,4_1)$ it is
\[ \iota_K(b) = c \qquad \iota_K(c)=b \qquad \iota_K(a) = a+x \qquad \iota_K(x) = x+e \qquad \iota_K(e)=e. \]
The figure-eight knot has two inequivalent strong inversions, related by mirroring. The maps associated to them have been computed by Dai, Mallick, and Stoffregen \cite[Example 2.26]{DMS:equivariant} and are
\[ \tau_K(b) = c \qquad \tau_K(c)=b \qquad \tau_K(a) = a+x \qquad \tau_K(x)=x \qquad \tau_K(e) = e \]
and
\[ \sigma_K(b) = c \qquad \sigma_K(c)=b \qquad \sigma_K(a) = a \qquad \sigma_K(x)=x+e \qquad \sigma_K(e) = e. \]
In this case the potential twist from changing the direction does not affect the $(\iota_K, \tau_K)$ or $(\iota_K, \sigma_K)$ chain homotopy equivalence class of the resulting complex, and therefore we may ignore it. We see that $\iota_K \tau_K$ fixes $b, c, a+x$ and interchanges $x$ and $x+e$, whereas $\iota_K \sigma_K$ fixes $b, c, x$ and interchanges $a$ and $a+x$. In particular we conclude that if $H$ is a Heegaard diagram for the figure-eight knot with either strong inversion we have
\[ \dim \HFKR(H, 1) = \dim \HFKR(H, 0) = \dim \HFKR (H, -1) =1. \]
Note that the author is not presently aware of an example of a thin knot for which, if one assumes that the $d_1$ differential is $(1+(\iota_K\tau_K)_*)\theta$, it does not follow that $\dim \HFKR(H, 0)=1$, although this may be an artifact of the fact there there exist relatively few computations of $\tau_K$ in the literature.

We next consider the equivariant connect sum of the torus knot $K=T_{3,4}$ with itself with a connected sum of the unique strong inversion on each summand. The equivariant connected sum is an operation on directed strongly invertible knots is due to Sakuma \cite{Sakumainvertible}; a nice picture appears in \cite[Definition 2.4]{DMS:equivariant}. In particular, this operation depends on the direction, but our choice of directions does not impact the following algebra. The knot Floer homology of the original knot has a chain complex with trivial differential 
\[\CFKhat(S^3,K,3) \simeq \mathbb F\langle x_{31} \rangle, \qquad \CFKhat(S^3,K,-3) \simeq \mathbb F_2\langle x_{32} \rangle\] \[ \CFKhat(S^3,K,2) \simeq \mathbb F_2\langle x_{21} \rangle\qquad \CFKhat(S^3,K,-2) \simeq \mathbb F_2\langle x_{22} \rangle \] \[\CFKhat(S^3,K, 0) = \mathbb F \langle x_0 \rangle.\] 
The maps $\iota_K$ and $\tau_K$ are both the unique Alexander-grading reversing chain isomorphism, fixing $x_0$ and exchanging $x_{j1}$ with $x_{j2}$ for $j=2,3$. We now consider the connected sum of the knot with itself. Ozsv{\'a}th and Szab{\'o} show that knot Floer homology has a K{\"u}nneth formula \cite[Section 7]{OSKnots}, and in particular $\CFKhat(K_1 \# K_1, i) \simeq \bigoplus_{i_1+i_2=i} \CFKhat(K_1,i_1) \otimes \CFKhat(K_2,i_2)$. Focusing on Alexander grading zero, which will be the most interesting, we see that
\[ \CFKhat(K\#K, 0) \simeq \bF \langle x_0x_0, x_{21}x_{22}, x_{22}x_{21}, x_{31}x_{32}, x_{32}x_{31} \rangle \]
Here concatenation denotes tensor product. The maps $\iota_K$ and $\tau_K$ do not have the same product formulas. Applying the tensor product formula for $\tau_K$ of \cite[Theorem 1.7]{DMS:equivariant}, we see that $\tau_K$ is the map interchanging the two factors in the tensor product. Applying the product formula for $\iota_K$ of \cite[Theorem 1.1]{ZemConnectedSums}, which requires the full knot Floer complex $CFK^-(T_{3,4})$ to compute, we see that $\iota_K$ is the map which sends $x_{22}x_{21} \mapsto x_{21}x_{22} + x_{31}x_{32}$ and otherwise exchanges the factors in the tensor product. This implies that the composition $\iota_K\tau_K$ is given by
\[ x_{21}x_{22} \leftrightarrow x_{21}x_{22}+x_{31}x_{32}\qquad \qquad x_0x_0, x_{22}x_{21}, x_{32}x_{31} \text{ fixed.}\]
We conclude that $\dim(\HFKR(H,0)) = 3$. A similar computation shows that in all other Alexander gradings $\iota_K =\tau_K$ and therefore $\dim(\HFKR(H,i)) = \dim(\HFKhat(K\#K,i))$. \end{example}

\begin{rem} \label{rem:speculative} The enthusiastic reader is cautioned that not all knot symmetries should be expected to admit a notion of real knot Floer homology. Guth and Manolescu's construction assumes an orientation-preserving symmetry on a 3-manifold $Y$ whose fixed set is nonempty and has codimension two; this means that one should not expect to be able to deal with (for example) freely two-periodic or strongly amphichiral knots. More subtly, if $H$ is a real Heegaard diagram for some symmetric knot $K$ with $w$ and $z$ fixed by $R$, then $R(H)$ is necessarily a Heegaard diagram for $K$ with its orientation reversed: the orientation of the surface and the order of the alpha and beta curves changes, but the basepoints remain in place. This suggests there is no real doubly-pointed Heegaard diagram for the lift $\widetilde{K}$ of a knot $K$ in its branched double cover, where the action preserves the orientation of the knot, and in particular no analog of the spectral sequence~\eqref{eqn:original} in real knot Floer homology. On the other hand, if one were to develop a notion of real Heegaard Floer knot homology for which the $z$ and $w$ basepoints may be away from the fixed set and interchanged by the action, it seems that it could be possible to study two-periodic knots; in which case one should obtain a spectral sequence whose $d_1$ differential sends $\HFKhat(S^3,K,i)$ to $\HFKhat(S^3,K,-i)$, fixing only $\HFKhat(S^3,K,0)$. We leave this development for the future.
\end{rem}

\section{Remarks on developments since the first appearance of this note} \label{sec:since-then}

For the reader's convenience, we conclude by mentioning some developments since this note first appeared. Most notably, a theory of real nice Heegaard diagrams has been developed independently by Lipshitz-Oszv{\'a}th \cite{LO:real-bordered} and Binns-Guth-Xiao \cite{BGX:sutured}, from which both groups confirm the conjectural computation of the $d_1$ differential of the spectral sequence appearing in Theorem \ref{thm:branched}. Moreover, Xiao \cite{Xiao:real-link} has studied real link Floer homology for strongly invertible links in general 3-manifolds and given a proof of invariance, which in particular shows the real knot Floer homology of strongly invertible knots in this paper is an invariant of the directed strongly invertible knot. Together with the work of \cite{BGX:sutured} this confirms the conjectural computation of the $d_1$ differential in Proposition \ref{prop:si}. This removes the speculative nature of the examples discussed in Section~\ref{sec:examples-knots}. Xiao's work together with code written by Z. Li \cite{Li:real-grid} produces many new examples, including of thin knots with $\dim{\HFKR}(S^3,K, \tau, 0)$ more than one.

Xiao further studies the real knot Floer homology of doubly periodic knots, producing \cite[Theorem 4.11]{Xiao:real-link} the spectral sequence promised in Remark~\ref{rem:speculative} of the form
\begin{equation} \label{eq:periodic} \HFKhat(S^3, K, 0) \otimes \mathbb F_2[\theta, \theta^{-1}] \rightrightarrows \HFKR(S^3,K) \otimes \mathbb F_2[\theta, \theta^{-1}].
\end{equation}
For completeness, we add the analysis of the $d_1$ differential of the spectral sequence \eqref{eq:periodic} in the presence of an $R$-symmetric family of almost complex structures achieving transversality as a final note. Following the argument of Proposition~\ref{prop:si}, we will see that this first differential is
\[ d_1 = (1 +(\iota_K \tau_K^{-1})_*)\theta = (1 +(\iota_K\tau_K\xi_K)_*)\theta = (1 +(\iota_K\xi_K \tau_K)_*)\theta\]
where $\tau_K \co \CFKhat(\cH) \rightarrow \CFKhat(\cH)$ is the action of a two-periodic knot on knot Floer homology defined in \cite[Section 8]{DMS:equivariant} and $\xi_K$ is as previously the Sarkar basepoint-moving involution.  To see this, we retrace the argument of Lemma~\ref{lem:iotaKtauK} for the periodic case. From \cite[Section 8.1]{DMS:equivariant}, for a two-periodic knot $K$ in $S^3$ with action $\tau$, the map $\tau_K$ induced by the action is the composition 
\[ \CFKhat(\cH) \xrightarrow{\eta_{\tau}} \CFKhat(\tau\cH) \xrightarrow{\eta_{\rho}} \CFKhat(\rho \tau \cH) \xrightarrow {\Phi{\rho \tau \cH, \cH}} \CFKhat(\cH). \]
Note that for a two-periodic knot, $\tau_K^2 \simeq \xi_K$, and in particular $\tau_K^{-1} \simeq \tau_K \circ \xi_K \simeq \xi_K \circ \tau_K$ \cite[Theorem 8.1]{DMS:equivariant}. We also remind the reader that the Sarkar involution $\xi_K$ is itself homotopic to $\eta_{\rho}^2$, and $\eta_{\rho}^4$ is homotopic to the identity. If $\cH = (H,J)$ is a real Heegaard diagram for a two-periodic knot in the sense of Xiao \cite[Section 2]{Xiao:real-link} together with an $R$-symmetric family of almost complex structures achieving transversality, $\tau \cH = \overline{\cH}$ precisely. We compute 
\begin{align*} 
\iota_K \circ \tau_K^{-1} &\simeq \iota_K \circ \xi_K \circ \tau_K \\
						&\simeq \eta_K \circ \Phi_{\rho\cH, \overline{\cH}} \circ \eta_{\rho}\circ \eta_{\rho}^2 \circ \Phi_{\rho \tau \cH, \cH} \circ \eta_{\rho} \circ \eta_{\tau} \\
						&\simeq \eta_K \circ \Phi_{\rho\cH, \overline{\cH}} \circ \eta_{\rho}^3 \circ \eta_{\rho} \circ \Phi_{\tau \cH, \rho\cH} \circ \eta_{\tau} \\
						& \simeq \eta_K \circ \Phi_{\rho\cH, \overline{\cH}}  \circ \Phi_{\overline{\cH}, \rho\cH} \circ \eta_{\tau} \\
						&\simeq \eta_K \circ \eta_{\tau}.
\end{align*}
We finally observe that $\eta_K\circ\eta_{\tau}$ is exactly the map produced by the geometry of the spectral sequence in the two-periodic knot case; since both $\eta_K$ and $\eta_{\tau}$ exchange basepoints, we do not pick up an additional factor of $sw$ as we do in the proof of Proposition~\ref{prop:si}. This concludes the argument.

This conclusion is perhaps unsurprising, as the argument of Lemma~\ref{lem:iotaKtauK} actually relates the $d_1$ differential in the strongly invertible case to $(1+(\iota_K\tau_K^{-1})_*)$ as well; the difference between the two analyses arises from the fact that for a strongly invertible knot $K$, the induced map $\tau_K$ is its own inverse up to homotopy, and for a periodic knot, the induced map $\tau_K$ has $\tau_K^{-1} \simeq \tau_K \xi_K \simeq \xi_K \tau_K$ \cite[Theorem 8.1]{DMS:equivariant}.

\bibliographystyle{custom}
\def\MR#1{}
\bibliography{biblio}

\end{document}